\documentclass[11pt,oneside]{article}

\usepackage[ruled]{algorithm2e}
\usepackage{bm,enumerate}
\usepackage{amssymb,amsfonts,amsthm}
\usepackage{amsmath}
\allowdisplaybreaks
\usepackage{mathtools,xfrac}
\usepackage{url,color}
\usepackage{graphicx}
\usepackage{enumitem}
\usepackage{appendix}
\usepackage{xcolor}
\usepackage[utf8]{inputenc}

\usepackage{multicol}
\usepackage[sort,numbers]{natbib}

\usepackage{verbatim}
\usepackage[colorlinks=true,breaklinks=true,bookmarks=true,urlcolor=blue,citecolor=blue,linkcolor=blue,bookmarksopen=false,draft=false]{hyperref}

\usepackage{tikz}
\usepackage{pgfplots}
\pgfplotsset{compat=1.3}

\usepackage{lineno}

\usepackage[margin=1in]{geometry}
\usepackage{authblk}
\makeatletter
\def\imod#1{\allowbreak\mkern10mu({\operator@font mod}\,\,#1)}
\makeatother
\usepackage[compact]{titlesec}
\usepackage{framed}

\theoremstyle{plain}
\newtheorem{theorem}{Theorem}[section]

\newtheorem{lemma}[theorem]{Lemma}
\newtheorem*{lemma*}{Lemma}

\newtheorem*{proposition*}{Proposition}
\newtheorem{corollary}[theorem]{Corollary}

\newtheorem*{observation*}{Observation}

\theoremstyle{definition}

\theoremstyle{remark}

\newtheorem*{remark*}{Remark}

\newtheorem*{example*}{Example}

\definecolor{gold}{rgb}{0.85,0.65,0}

\newcommand{\ip}[2]{\left\langle #1 , #2 \right\rangle}    %

\def\beq{\begin{equation}}
\def\eeq{\end{equation}}

\def\fnote#1{\footnote}

\def\cO{{\cal O}}

\newcommand{\bbE}{\mathbb{E}}

\newcommand{\bbR}{\mathbb{R}}

\DeclareMathOperator{\dom}{dom}

\begin{document}
\title{The Randomized Block Coordinate Descent Method in the H{\"o}lder Smooth Setting}
\author{Leandro Farias Maia}
\author{David Huckblerry Gutman}
\affil[1]{Texas A$\&$M University}

\maketitle

\begin{abstract}
This work provides the first convergence analysis for the Randomized Block Coordinate Descent method for minimizing a function that is both H\"older smooth and block H\"older smooth. Our analysis applies to objective functions that are non-convex, convex, and strongly convex. For  non-convex functions, we show that the expected gradient norm reduces at an $O\left(k^{\frac{\gamma}{1+\gamma}}\right)$ rate, where $k$ is the iteration count and $\gamma$ is the H\"older exponent. For convex functions, we show that the expected suboptimality gap reduces at the rate $O\left(k^{-\gamma}\right)$. In the strongly convex setting, we show this rate for the expected suboptimality gap improves to $O\left(k^{-\frac{2\gamma}{1-\gamma}}\right)$ when $\gamma>1$ and to a linear rate when $\gamma=1$. Notably, these new convergence rates coincide with those furnished in the existing literature for the Lipschitz smooth setting. 
\end{abstract}

\section{Introduction}

In this article, we provide non-asymptotic convergence rates for the Randomized Block Coordinate Descent (RBCD) method when applied to the problem
\begin{equation}\label{eq:problem}
f^{*} := \min_{x\in\mathbb{R}^d} f(x)
\end{equation}
where the objective function $f:\bbR^d\to\bbR$ is H\"older smooth, a generalization of the standard (Lipschitz) smoothness, and block H\"older smooth. Formally, the continuously differentiable function, $f$, is said to be \emph{H\"older smooth} when its gradient, $\nabla f$, is H\"older continuous, i.e. there exist $L>0$ and $\gamma\in (0,1]$ guaranteeing
\begin{equation}\label{eq:holder}
\|\nabla f(y)-\nabla f(x)\|\leq L\|y-x\|^\gamma\quad\text{for all }x,y\in\bbR^d.
\end{equation}

The popularity of block coordinate methods owes to their fitness for large-scale optimization problems emerging from applications in machine learning and statistics. Essentially, randomized block coordinate descent is a (random) block-wise adaptation of gradient descent. Instead of updating all coordinates simultaneously, the randomized block coordinate descent method updates a single, randomly selected coordinate block using only that block's partial gradient. The computational economy of these block gradient updates, relative to full gradient updates, are what make the randomized block coordinate  descent method especially attractive for large-scale problems. Given an initial point $x^0$, this cheap iterate update rule is somewhat more generally realized as
\begin{equation}\label{eq:BCD}
x^k=x^{k-1}-t_k\cdot P_{i_k}\nabla f(x^{k-1}), \quad k=1,2,\ldots
\end{equation}
where $t_k>0$, $i_k$ is selected randomly from $\{1,\ldots,m\}$, and $P_1,\ldots,P_m\in\bbR^{d\times d}$ are orthogonal projection matrices onto orthogonal subspaces that sum to $\bbR^d$. The``block coordinate" name originates from the archetypal choice for the orthogonal subspaces projected onto: spans of collections of coordinate vectors.

For coordinate descent methods, and indeed a preponderance of first-order methods, the intimate relationship between the selection of step-sizes and $\nabla f$'s regularity determines their convergence rates \cite{GlobConvGradHolder,Sh-Dvu-Gasn, NestUnivGradMethod, SmoothBanditOptGenHoldSpa, Nemi-Yur, Nemi-Ark-1983, MultScaZOOptSmFunc,GenHolSmooConvRatFollSpecRatAssGroBoun, Gutman22}. Both Bredies \cite{Bredies08} and Yashtini \cite{Yashtini16} study the interplay between step-size selection and convergence for gradient descent applied to \eqref{eq:problem} in the H\"older smooth regime. Bredies \cite{Bredies08} established a $O\left(1/k^\gamma\right)$ convergence rate of the proximal gradient method, a generalization of the gradient descent method, for convex composite minimization. On the other hand, Yashtini \cite{Yashtini16} established that, given an appropriate step-size selection, gradient descent converges at a  $O\left(1/k^{\frac{\gamma}{1+\gamma}}\right)$ for non-convex, H\"older smooth objective functions. 

We are unaware of any studies of the randomized block coordinate descent method that assume H\"older smoothness or its block-wise adaptation, block H\"older smoothness.  We say that the continuously differentiable function, $f$, is \emph{block H\"older smooth} if for each $i=1,\ldots,m$, there exists $L_i>0$ such that
\begin{equation}\label{eq:Holder-block}
\|\nabla f(x+P_i u)-\nabla f(x)\|\leq L_i\|P_i u\|^\gamma\text{ for all }u\in\bbR^d.
\end{equation}
The seminal articles \cite{Nesterov12,Richtarik14} studying the randomized block coordinate descent method all make the more restrictive assumption that the gradient is Lipschitz continuous. 
Recently, inspired by the work of both Bredies \cite{Bredies08} and Yashtini \cite{Yashtini16}, Gutman and Ho-Nguyen \cite{Gutman22} produced a convergence analysis for the cyclic block coordinate descent method assuming H\"older and block H\"older smoothness in both the convex and non-convex settings. Thus, the goal of this paper is to extend this analysis to the more popular randomized block coordinate descent method in the non-convex, convex, and even strongly convex settings. 

We conclude this introduction with an outline of our article that includes a high-level description of each of our main contributions. This article is structured into four primary sections:
\begin{itemize}
	\item \textit{Section} \ref{sec:algorithm}: In this section, we introduce our RBCD step-size selection for H\"older smooth objective functions as well as the attendant notation. We also introduce two key lemmata (Lemmata \ref{lemma:duality} and \ref{lemma:block_descent}) that support our analyses.
	\item \textit{Section} \ref{sec:nonconvex}: In this section, we present our convergence analysis for general, possibly non-convex objective functions satisfying H\"older and block H\"older smoothness conditions.  For these objectives, our proposed step-size ensures RBCD shrinks the expected gradient norm at a $O\left(1/k^{\frac{\gamma}{1+\gamma}}\right)$ rate (Theorem \ref{thm.conv.rate.general}). 
	\item \textit{Section} \ref{sec:convex}: In this section, we present our convergence analysis under the further assumption that the objective function is convex. In this setting, RBCD with our step-size shrinks the expected suboptimality gap at a $O\left(1/k^{\gamma}\right)$ rate for non-strongly convex objective functions (Theorem \ref{thm.conv.rate.convex}). Notably, our rates for these objective functions coincide with those of \cite{Nesterov12} when $\gamma=1$, or equivalently, when the objective is $L$-smooth. 
	\item \textit{Section} \ref{sec:strong-convex}: In this section, we present our analysis under the further assumption that the objective function is strongly convex. This analysis depends upon the value of the H\"older exponent, $\gamma$. When $\gamma=1$, we show RBCD converges at a linear rate (Theorem \ref{thm.conv.rate.strong-convex}). When $\gamma \in (0,1)$, we obtain a $\mathcal{O}\left( 1/k^{\frac{2\gamma}{1-\gamma}} \right)$ rate of convergence (Theorem \ref{thm.conv.rate.strong-convex}). Moreover, we show that our sublinear rates converge to our linear rates as $\gamma\to 1$ (Corollary \ref{cor:str-cvx.interpolation}). As for convex objectives, our rates for strongly convex objectives coincide with those of \cite{Nesterov12} when the objective is $L$-smooth. 
\end{itemize}
\section{Notation and Step-size Selection for RBCD Under H\"older Smoothness}
\label{sec:algorithm}

This short section introduces the notation necessary for all of this article's developments, and  the H\"older smoothness-based step-size selection for the RBCD method. It also exhibits two lemmata, Lemmata \ref{lemma:duality} and \ref{lemma:block_descent}, that are used throughout the paper to aid the convergence analysis of the proposed method.

Our step-size selection is an adaptation of that used for the cyclic block setting from \cite{Gutman22} to the immensely more popular randomized block setting. Thus, our notation is a synthesis of that article's notation and the notation of \cite{Nesterov12}, one of the canonical works on randomized coordinate descent. We let $\tilde{L}:=\{L_1,\ldots,L_m\}$ denote the set of the block H\"older smoothness constants. For $\alpha\in\mathbb{R}$, we define the new constant, $S_{\alpha}(f)$, as
\[
S_{\alpha}(f) = \sum_{i=1}^{m}L_i^{\alpha}.
\]
When $f$ is clear from context, we will simply write $S_\alpha$. For the sake of concision, we set $\nabla_i f(x):=P_i \nabla f(x)$ for all $x\in\bbR^d$ and $1\leq i\leq m$. We adopt the notation, $\nu:=\frac{1+\gamma}{\gamma}>1$, because the quantity $\frac{1+\gamma}{\gamma}$  frequently appears in our analysis.

Much of our analysis is framed in terms of $\tilde{L}$-weighted $q$-norms on $\bbR^d$, $\|\cdot\|_{\alpha,q}$. Given $\alpha\in\bbR$ and $q\geq 1$, we let 
\begin{equation}\label{eq:weighted_norm}
\|x\|_{\alpha,q} := \left[ \sum_{i=1}^{m}L_{i}^\alpha\|P_i x\|^q  \right]^{1/q}.
\end{equation}
When $\alpha=0$, $\|\cdot\|_{\alpha,q}$ reduces to the standard $q$-norm, which we write as $\|\cdot\|_q$. For simplicity, we let $\|\cdot\|:=\|\cdot\|_2$.
We bare three important notes about these weighted norms. First, $\|\cdot\|_{\alpha,q}$ generalizes the norm
\[
x\mapsto\left[\sum_{i=1}^{m}L_{i}^\alpha\|P_i x\|^2  \right]^{1/2},
\]
which plays a starring role throughout Nesterov's classical analysis of randomized coordinate descent methods from \cite{Nesterov12} in the block Lipschitz smooth setting. The flexibility provided by changing the exponents $2$ and $1/2$ to $q$ and $1/q$, respectively, is critical to capturing our more general H{\"o}lderian convergence rates.  Additionally, the parameter $\alpha$ permits us to simultaneously achieve RBCD's convergence rates for two different, common random block selection schemes:
\begin{itemize}
\item[i)] $\alpha=0$ corresponds to selecting the blocks uniformly at random;
\item[ii)] $\alpha=1$ corresponds to selecting the $i$-th block with probability $L_i/\sum_{i=1}^m L_i$.
\end{itemize}
Finally, these weighted norms possess natural duality relationships and equivalences to the Euclidean norm, which we liberally use throughout our analysis and summarize in the below lemma.
\begin{lemma}[$(\alpha,q)$-Norm Duality and Equivalences]\label{lemma:duality}
Let $\alpha\in\bbR$, $p\in[1,\infty]$, and $q$ be the H{\"o}lder conjugate of $p$, i.e. $q:=\frac{p}{p-1}$. The following hold for $\|\cdot\|_{\alpha,p}$:
\begin{enumerate}
	\item The Cauchy-Schwarz inequality
\begin{equation}\label{eq:cauchy-schwarz}
|\ip{x}{y}|\leq \|x\|_{\alpha,q}\|y\|_{-\alpha\frac{p}{q},q}
\end{equation}
holds for all $x,y\in\bbR^d$. Equality is obtained if and only if $x=0$ or 
\[
y=c\cdot\sum_{i=1}^m L_i^\alpha \|P_i x\|^{p-2}P_i x
\] 
for some $c\in\bbR$. Consequently, $\|\cdot\|_{-\alpha\frac{p}{q},q}$ is the dual norm of $\|\cdot\|_{\alpha,q}$.
	\item If $p\geq 2$ and $\alpha,\beta\in\bbR$ then the norms $\|\cdot\|_{\alpha,p}$ and $\|\cdot\|_{\beta,2}$ satisfy 
\[
\left( \max_{1\leq i\leq m}L_i^{\frac{\alpha}{p}-\frac{\beta}{2}}\right)\cdot  \|x\|_{\beta,2} \geq \|x\|_{\alpha,p} \geq \left( m^{\frac{1}{p}-\frac{1}{2}}\cdot \min_{1\leq i\leq m}L_i^{\frac{\alpha}{p}-\frac{\beta}{2}}\right) \cdot \|x\|_{\beta,2} 
\]
for all $x\in\bbR^d$.
\end{enumerate} 
\end{lemma}

\noindent We defer the proof of this lemma to the appendix (Appendix \ref{app:duality}) to maintain the focus of our exposition. 

With all of the article's requisite notation in hand, we may introduce our main algorithm (Algorithm \ref{algo.main}), and describe an associated descent lemma (Lemma \ref{lemma:block_descent}). We note that our step-size, $-\|\nabla_{i}f(x^k)\|^{\nu-2}/L_i^{\nu-1}$, coincides with that proposed in \cite{Nesterov12} when $\gamma=1$. Thus, we may view it as a generalization that accounts for the use of block H\"older smoothness in the place of standard block smoothness.

\begin{algorithm}[H]
\caption{Randomized Block Coordinate Descent Method (RBCD)}\label{algo.main}
	\KwData{$x^0 \in \dom(f)$, $\alpha\in[0,1]$}
	\For{k=0,1,2,\ldots}{
	Choose
		\[
		i_k\sim \left(p_1,\ldots,p_m\right):=\left(\frac{L_{1}^{\alpha}}{\sum_{j=1}^{m}L_{j}^{\alpha}},\ldots,\frac{L_{m}^{\alpha}}{\sum_{j=1}^{m}L_{j}^{\alpha}}\right);
	\]
	\noindent Update block $i_k$ of $x^k$ according to
		\begin{equation}\label{eq:main_update}
			x^{k+1}:=x^{k}-\frac{\|\nabla_{i}f(x^k)\|^{\nu-2}}{L_i^{\nu-1}}\cdot \nabla_i f(x^k).
		\end{equation}
		}	
\end{algorithm}

A special case of the main descent lemma of \cite{Bredies08}, derived in \cite{Gutman22}, plays the same role in our analysis that it played for the cyclic block analysis in \cite{Gutman22}. We directly quote this special case from \cite{Gutman22} below.

\begin{lemma}
\label{lemma:block_descent}[Block H{\"o}lder Descent Lemma, \cite{Gutman22}, Lemma 1] Let $f:\mathbb{R}^{d}\rightarrow R$ be a function that satisfies the block H\"older smoothness condition. For any $i$, $1\leq i \leq m$,
\begin{equation}\label{eq:upper-model}
    f(x+U_iy) \leq f(x) + \ip{\nabla_i f(x)}{U_iy} + \frac{L_i}{1+\gamma}\|U_iy\|^{1+\gamma}_2.
\end{equation}
Moreover, if $x^+$ is the minimizer of the right-hand side of \eqref{eq:upper-model}, i.e.
\[
x^+=x-\frac{\|\nabla_{i}f(x)\|^{\nu-2}}{L_i^{\nu-1}}\cdot \nabla_i f(x),
\]
then
\[
f(x)-f(x^+)\geq\frac{1}{\nu L_i^{\nu-1}}\|\nabla_i f(x)\|^\nu.
\]
\end{lemma}

\section{Convergence Analysis: General Objectives}
\label{sec:nonconvex}

In this section, we layout our convergence rate analysis for non-convex objectives satisfying H\"older smoothness \eqref{eq:holder} and block H\"older smoothness \eqref{eq:Holder-block}. We will present the main convergence theorem (Theorem \ref{thm.conv.rate.general}) after we elaborate our key Sufficient Decrease Lemma (Lemma \ref{lemma:sufficient-decrease}). This lemma facilitates all of our convergence analyses.

\begin{lemma}\label{lemma:sufficient-decrease}
\emph{(Sufficient Decrease)}  Let $\{x_n\}_{n=0}^{\infty}$ be the sequence generated by RBCD (Algorithm \ref{algo.main}). If $f$ satisfies our H{\"o}lder smoothness \eqref{eq:holder} and block H\"older smoothness \eqref{eq:Holder-block} assumptions, then
\begin{equation}\label{eq:sufficient-decrease}
\frac{1}{\nu S_\alpha(f)}\|\nabla f(x_k)\|_{\alpha+1-\nu,\nu}^\nu\leq f(x_k)-\bbE\left[f(x^{k+1})\bigg| x^k\right]
\end{equation}
holds for all $k\geq 0$.
\end{lemma}

\begin{proof}
Expanding the expectation-defining sum, and applying the block descent lemma (Lemma \ref{lemma:block_descent}), we compute
\begin{align*}
f(x_k)-\bbE\left[f(x^{k+1})\bigg| x^k\right]&=\bbE\left[f(x_k)-f(x^{k+1})\bigg| x^k\right]\\
&=\sum_{i=1}^{m}\left(\frac{L_i^{\alpha}}{\sum_{j=1}^{m}L_{j}^{\alpha}}\right)\cdot\left[ f(x_k)-f\left(x^{k}-\frac{\|\nabla_{i}f(x^k)\|^{\nu-2}}{L_i^{\nu-1}}\cdot \nabla_i f(x^k)\right)\right]\\
&\stackrel{\text{Lemma }\ref{lemma:block_descent}}\geq \frac{1}{\nu S_\alpha}\sum_{i=1}^{m}L_i^{\alpha+1-\nu}\|\nabla_{i}f(x^k)\|^\nu\\
&=\frac{1}{\nu S_\alpha}\|\nabla f(x_k)\|_{\alpha+1-\nu,\nu}^\nu.
\end{align*}
Rearranging the inequality and taking total expectations completes the proof.
\end{proof}

Next, we present the centerpiece of this section, our main convergence theorem for non-convex objective functions.

\begin{theorem}[RBCD Convergence: General Objective Functions]\label{thm.conv.rate.general}
Let $\{x_n\}_{n=0}^{\infty}$ be the sequence generated by RBCD (Algorithm \ref{algo.main}). If $f$ satisfies our H{\"o}lder smoothness \eqref{eq:holder} and block H\"older smoothness \eqref{eq:Holder-block} assumptions, then
\[
\min_{0\leq j \leq k}\mathbb{E}\left[\|\nabla f(x^j)\|_{1+\alpha-\nu,\nu}\right] \leq  \left(\nu S_\alpha(f)\right)^{\frac{1}{\nu}} \cdot\left(\frac{f(x^0)-f^*}{k+1}\right)^{\frac{1}{\nu}}= \cO\left( k^{-\frac{1}{\nu}}\right)
\]
holds for all $k\geq 0$. Consequently, we have the convergence rate measured in the norm $\|\cdot\|_{\beta,2}$,
\[
\min_{0\leq j \leq k}\mathbb{E}\left[\|\nabla f(x^j)\|_{\beta,2}\right] \leq\left(\frac{\displaystyle \max_{1\leq i\leq m} L_i^{\frac{\beta}{2}-\frac{1+\alpha-\nu}{\nu}}}{m^{\frac{\nu-2}{2\nu}}}\right)\cdot\left(\nu S_\alpha(f)\right)^{\frac{1}{\nu}} \cdot\left(\frac{f(x^0)-f^*}{k+1}\right)^{\frac{1}{\nu}}= \cO\left( k^{-\frac{1}{\nu}}\right),
\]
holds for all $k\geq 0$.
\end{theorem}

\begin{proof}
For each $k\geq 0$, observe that
\begin{align*}
\min_{0\leq j\leq k}\mathbb{E}\left[\|\nabla f(x^j)\|_{\alpha+1-\nu,\nu}\right]^\nu
&\leq \frac{1}{(k+1)}\cdot\sum_{j=0}^k\mathbb{E}\left[\|\nabla f(x^j)\|_{\alpha+1-\nu,\nu}\right]^\nu\\
&\leq \frac{1}{(k+1)}\cdot\sum_{j=0}^k\mathbb{E}\left[\|\nabla f(x^j)\|_{\alpha+1-\nu,\nu}^\nu\right]\\
&\leq \nu S_\alpha\cdot \frac{1}{(k+1)}\cdot\sum_{j=0}^k  \left(\bbE\left[f(x^j)\right]-\bbE\left[f(x^{j+1})\right]\right)\\
&=\nu S_\alpha\cdot\frac{f(x^0)-\bbE[f(x^{k+1})]}{k+1}\\
&\leq\nu S_\alpha\cdot\frac{f(x^0)-f^*}{k+1}
\end{align*} 
where we apply Jensen's inequality to the expectation operator for the convex function $x\mapsto x^\nu$ in the second line, and Lemma \ref{lemma:sufficient-decrease} in  the third line. Taking $\nu$-th roots of both sides of the resultant inequality above, produces our first result. 

The result in terms of the $\|\cdot\|_{\beta,2}$ follows immediately from Lemma \ref{lemma:duality}.
\end{proof}

\section{Convergence Analysis: Convex Objectives}
\label{sec:convex}

In this section, we forward our convergence analysis of RBCD (Theorem \ref{thm.conv.rate.convex}) for convex objective functions that are both H\"older \eqref{eq:holder} and block H\"older \eqref{eq:Holder-block} smooth. First, we  present a Technical Recurrence Lemma (Lemma \ref{lemma_poly}) that helps produce our convergence rates in this section, and a subset of the convergence rates for strongly convex objective functions in the sequel. Next, we exhibit a techical lemma (Lemma \ref{lemma:nesterov-holder-cvx-str}) that permits us to express our rates in terms of the diameter of the initial sublevel set. Finally, the section concludes with our main convergence theorem (Theorem \ref{thm.conv.rate.convex}) and a comparison of these rates to those furnished for smooth and convex functions in \cite{Nesterov12}. 

As promised, we begin this section with a Technical Recurrence Lemma that supports the derivation of our convergence rates.

\begin{lemma}
\label{lemma_poly}(Technical Recurrence, \cite[Chapter~2, Lemma~6]{Polyak_book})\label{lemma:recurrence}
If $\{A_k\}_{k\geq 0}$ is a non-negative sequence of real numbers satisfying the recurrence
\[
A_{k+1}\leq A_k - \theta A_{k}^r
\]
for some $\theta \geq 0$ and $r> 1$, then
\[
A_k \leq \frac{A_0}{(1+(r-1)\theta A_0^{r-1}k)^{\frac{1}{r-1}}}
\]
\end{lemma}

The following lemma permits us to express our convergence rates here and in the sequel section in terms of the initial sublevel set's diameter.

\begin{lemma}\label{lemma:nesterov-holder-cvx-str}
Under the Block H\"older Smoothness assumption \eqref{eq:Holder-block} and coercivity of $f$, $f$ satisfies
\begin{align*}
f(x)-f^* \leq\left(\frac{\nu S_\alpha(f)}{2}\right)^{\frac{1}{\nu-1}}\cdot\left(\frac{\nu-1}{\nu}\right)\cdot R(x)_{(1+\alpha-\nu)(1-\nu),\frac{\nu}{\nu-1}}^{\frac{\nu}{\nu-1}}
\end{align*}
for all $x\in\bbR^d$, where $R_{\beta,q}(x):=\max\left\{\|y-x^*\|_{\beta,q}:f(x^*)=f^*,f(y)\leq f(x)\right\}<\infty$.
\end{lemma}
\noindent We defer the proof of this lemma to the appendix (Appendix \ref{app:level-set}).

Finally, equipped with these tools, we present and prove the theorem that establishes RBCD's convergence rate for convex functions. Afterward, we explain its relationship to its analogue for smooth and convex functions in \cite{Nesterov12}.

\begin{theorem}[RBCD Convergence: Convex Objective Functions]\label{thm.conv.rate.convex}
Let $\{x_n\}_{n=1}^{\infty}$ be the sequence generated by RBCD (Algorithm \ref{algo.main}). If $f$ is a convex and coercive function that satisfies our H{\"o}lder smoothness \eqref{eq:holder} and block H\"older smoothness \eqref{eq:Holder-block} assumptions, then
\[
\mathbb{E}[f(x^{k})]- f^*\leq\frac{\left(\nu S_\alpha(f) R_{(1+\alpha-\nu)(1-\nu),\frac{\nu}{\nu-1}}(x^0)^\nu\right)^{\frac{1}{\nu-1}}(\nu-1)}{\left[2\nu^{\nu-1}+(\nu-1)^\nu k\right]^{\frac{1}{\nu-1}}}=\cO\left(k^{-\frac{1}{\nu-1}}\right)
\]
where $R_{\beta,q}(x^0):=\max_{y}\left\{\|y-x^*\|_{\beta,q}:f(x^*)=f^*,f(y)\leq f(x^0)\right\}<\infty$
\end{theorem}

\begin{proof}
The bulk of this proof centers on an application of the Technical Recurrence Lemma (Lemma \ref{lemma:recurrence}). In the context of that lemma, we let $A_i = \mathbb{E}[f(x^{i})]- f^*$ for each $i\geq 0$. By definition, $A_i\geq 0$ for each $i\geq 0$. To simplify notation, we let $R:=R_{(1+\alpha-\nu)(1-\nu),\frac{\nu}{\nu-1}}(x^0)$.

With this notation, we may restate the sufficient decrease inequality \eqref{eq:sufficient-decrease} of Lemma \ref{lemma:sufficient-decrease} as
\begin{equation*}
\bbE\left[\|\nabla f(x_k)\|_{\alpha+1-\nu,\nu}^\nu\right]\leq \nu S_\alpha\cdot\left(A_k-A_{k+1}\right),
\end{equation*}
or, equivalently,
\begin{equation}\label{eq:cvx.recur}
A_{k+1}\leq A_k-\frac{1}{\nu S_\alpha}\bbE\left[\|\nabla f(x_k)\|_{\alpha+1-\nu,\nu}^\nu\right].
\end{equation}
Thus, to apply the Technical Recurrence Lemma (Lemma \ref{lemma:recurrence}) we need only bound the expectation on the right below by $A_k^\nu$. By the Cauchy-Schwarz inequality (Lemma \ref{lemma:duality}) for $\|\cdot\|_{1+\alpha-\nu,\nu}$ and its dual $\|\cdot\|_{(1+\alpha-\nu)(1-\nu),\frac{\nu}{\nu-1}}$, we achieve for any optimum $x^*$, that
\begin{align*}
f(x^k)-f^*&\leq \ip{x^k-x^*}{\nabla f(x^k)}\\
&\leq\|x^k-x^*\|_{(1+\alpha-\nu)(1-\nu),\frac{\nu}{\nu-1}}\|\nabla f(x^k)\|_{1+\alpha-\nu,\nu}\\
&\leq R\|\nabla f(x^k)\|_{1+\alpha-\nu,\nu}.
\end{align*}
Raising each side of the above inequality to the power $\nu$, taking expectations, and applying  Jensen's inequality to the convex function $x\mapsto x^{\nu/2}$, we conclude
\begin{equation*}
A_k^\nu=\left(\bbE[f(x^k)]-f^*\right)^\nu\leq \bbE\left[\left(f(x^k)-f^*\right)^\nu\right]\\
\leq R^\nu\bbE\left[\|\nabla f(x^k)\|_{1+\alpha-\nu,\nu}^\nu\right].
\end{equation*}
Stringing together our work above, equation \eqref{eq:cvx.recur} yields the recurrence
\[
A_{k+1}\leq A_k-\frac{1}{\nu S_\alpha R^\nu} A_k^\nu
\]
for each $k\geq 0$.
We are now permitted to apply the Technical Recurrence Lemma (Lemma \ref{lemma:recurrence}) with  $r=\nu$ and $\theta=\frac{1}{\nu S_\alpha R^\nu}$ to produce
\begin{equation*}
\mathbb{E}[f(x^{k})]- f^*\leq\frac{f(x^{0})- f^*}{\left[1+(\nu-1)\cdot\theta\cdot\left(f(x^{0})- f^*\right)^{\nu-1}\cdot k\right]^{\frac{1}{\nu-1}}}.
\end{equation*}
We dedicate the remainder of this proof to simplifying this convergence bound. By factoring $f(x^{0})- f^*$ out of both the numerator and denominator, we may equivalently write the right-hand side of this bound as
 \begin{equation*}
\frac{1}{\left[\frac{1}{\left(f(x^{0})- f^*\right)^{\nu-1}}+(\nu-1)\cdot\theta\cdot k\right]^{\frac{1}{\nu-1}}}.
\end{equation*}

By considering $x=x^0$ in Lemma~\ref{lemma:nesterov-holder-cvx-str}, raising both sides to the power $\nu-1$ and applying the norm equivalence inequality from Lemma \ref{lemma:duality}, we further see
\begin{align*}
\left[f(x^0)-f(x^*)\right]^{\nu-1}&\leq \left(\frac{\nu S_\alpha}{2}\right)\cdot\left(\frac{\nu-1}{\nu}\right)^{\nu-1}\cdot R^\nu
\end{align*}
so the right-hand side of our bound simplifies to
\[
\left[\frac{1}{\left(\frac{\nu S_\alpha}{2}\right)\cdot\left(\frac{\nu-1}{\nu}\right)^{\nu-1}\cdot R^\nu}+(\nu-1)\cdot\left(\frac{1}{\nu S_\alpha R^\nu}\right)\cdot k\right]^{-\frac{1}{\nu-1}}=\frac{\left(\nu S_\alpha R^\nu\right)^{\frac{1}{\nu-1}}(\nu-1)}{\left(2\nu^{\nu-1}+(\nu-1)^\nu k\right)^{\frac{1}{\nu-1}}},
\]
which concludes the proof.
\end{proof}
Notably, our rate matches that provided by Nesterov in the standard block smooth setting, i.e. when $\nu=2$ we recover the convergence rate,
\[
\bbE[A_{k+1}] \leq \frac{2}{k+4}S_\alpha (f) R^2(x_1) = \mathcal{O}(k^{-1}),
\]
from \cite{Nesterov12}.

\section{Convergence Analysis: Strongly Convex Objectives}
\label{sec:strong-convex}

In this final section, we conclude the paper with a convergence analysis of RBCD (Algorithm \ref{algo.main}) for strongly convex objective functions that are both H\"older \eqref{eq:holder} and block H\"older \eqref{eq:Holder-block} smooth. We say that $f:\mathbb{R}^d\to \mathbb{R}$ is $\sigma$-strongly convex with respect to the norm $\|\cdot\|_{1-\alpha,2}$, where $\sigma>0$, if
\begin{equation}\label{eq:str-cvx}
f(y) \geq f(x) +\ip{\nabla f(x)}{y-x} + \frac{1}{2}\sigma \|x-y\|_{1-\alpha,2}^2
\end{equation}
for all $x,y\in\bbR^d$. The section begins with our main theorem (Theorem \ref{thm.conv.rate.strong-convex}), which provides rates in both the $L$-smooth and H\"older smooth settings. Next, we compare these rates with those in the previous section and \cite{Nesterov12}. Finally, we show that the smooth setting's linear rate is achieved in the limit as $\nu\to 2$, or equivalently, $\gamma\to 1$ (Corollary~\ref{cor:str-cvx.interpolation}). 

Without further ado, we present our main convergence theorem for strongly convex objective functions.

\begin{theorem}[RBCD Convergence: Strongly Convex Objective Functions]\label{thm.conv.rate.strong-convex}
Let $\{x_n\}_{n=1}^{\infty}$ be the sequence generated by RBCD (Algorithm \ref{algo.main}). Suppose that $f:\bbR^d\to\bbR$ is $\sigma$-strongly convex and satisfies both the H{\"o}lder and block H{\"o}lder  smoothness assumptions \eqref{eq:holder} and \eqref{eq:Holder-block}. The following hold:
\begin{enumerate}
\item (Linear Rate - Smooth Setting) If $\nu=2$, i.e. $f$ is smooth, then
\begin{equation}\label{eq:linear-rate}
\bbE[f(x_k)] -f^* \leq \left(1-\frac{\sigma}{S_\alpha(f)} \right)^k \cdot \frac{S_\alpha(f)^{\frac{1}{\nu-1}}(\nu-1)R(x^0)^{\frac{\nu}{\nu-1}}}{\nu^{\frac{\nu-2}{\nu}}2^{\frac{1}{\nu-1}}}=\cO\left(\exp\left(-\frac{\sigma}{S_\alpha(f)}k\right)\right).
\end{equation}
\item (Sublinear Rate - H{\"o}lder Smooth Setting) If $\nu>2$, i.e. $f$ is H{\"o}lder smooth but not smooth, then
\begin{equation*}
 \mathbb{E}[f(x^{k})]- f^* \leq \frac{C_0}{\left(C_1+C_2 k\right)^{\frac{2}{\nu-2}}}=\cO\left(k^{-\frac{2}{\nu-2}}\right),
\end{equation*}
where
\begin{equation*}
\begin{gathered}
C_0=(2\nu S_{\alpha}(f))^{\frac{2}{\nu-2}}m^{\frac{1}{\nu}}(\nu-1)R(x^0)^{\frac{\nu}{\nu-1}},\quad C_1=2^{\frac{\nu-2}{2(\nu-1)}}m^{\frac{\nu-2}{2\nu}} S_{\alpha}(f)^{\frac{\nu-1}{\nu}}\nu^{\frac{\nu^2-2\nu+4}{2\nu}}\\
C_2=R(x^0)^{\frac{\nu(\nu-2)}{2(\nu-1)}}(\nu-1)^{\frac{\nu-2}{2}}(\nu-2)(2\sigma)^{\frac{\nu}{2}}\displaystyle \min_{1\leq i\leq m} L_i^{\frac{(\alpha+1)(2-\nu)}{2\nu}}
\end{gathered}
\end{equation*}
\end{enumerate}
\end{theorem}

\begin{proof}
Let $R:=R_{(1+\alpha-\nu)(1-\nu),\frac{\nu}{\nu-1}}(x^0)$ to simplify notation. Both parts of the theorem speedily follow from the recurrence
\begin{equation}\label{eq:recur-str-cvx}
A_{k+1}\leq A_{k} - A_{k}^{\frac{\nu}{2}} \cdot \left( \frac{(2\sigma)^{\frac{\nu}{2}}\displaystyle \min_{1\leq i\leq m} L_i^{\frac{(\alpha+1)(2-\nu)}{2\nu}}}{\nu S_{\alpha}m^{\frac{1}{2}-\frac{1}{\nu}}}\right)
\end{equation}
where $A_i = \mathbb{E}[f(x^{i})]- f^*$ for each $i\geq 0$. After establishing \eqref{eq:recur-str-cvx}, we will separately show how each of the Theorem's two parts result from it. 

As in the proof of convergence for non-strongly convex functions, the sufficient decrease inequality \eqref{eq:sufficient-decrease} of Lemma \ref{lemma:sufficient-decrease} implies \eqref{eq:cvx.recur}, which we recall is
\[
A_{k+1}\leq A_k-\frac{1}{\nu S_\alpha}\bbE\left[\|\nabla f(x_k)\|_{\alpha+1-\nu,\nu}^\nu\right].
\]
Glancing at \eqref{eq:recur-str-cvx} and this latest inequality, it becomes immediately clear that we ought to bound $\bbE\left[\|\nabla f(x_k)\|_{\alpha+1-\nu,\nu}^\nu\right]$ below by $A_{k}^{\frac{\nu}{2}}=\bbE\left[ f(x_k)-f^*\right]^{\nu/2}$, appropriately scaled. To this end, strong convexity now makes it's main appearance. Using the standard argument of fixing $x\in\bbR^d$ in $\sigma$-strong convexity's defining inequality \eqref{eq:str-cvx} and minimizing it over $y\in\bbR^d$, we achieve the Polyak-\L ojasiewicz (PL) inequality
\[
\frac{1}{2\sigma} \left(\|\nabla f(x)\|_{1-\alpha,2}^*\right)^2 \geq f(x)-f^* .
\]
Setting $x=x_k$, raising both sides to the power $\nu/2$, and then taking expectations, we see 
\begin{equation*}
\frac{1}{(2\sigma)^{\frac{\nu}{2}}} \bbE\left[\left(\|\nabla f(x_k)\|_{1-\alpha,2}^*\right)^\nu \right] \geq \bbE\left[(f(x_k)-f^*)^{\nu/2}\right],
\end{equation*}
which, by Jensen's inequality applied to the convex function $x\mapsto x^{\nu/2}$, produces
\begin{equation}\label{eq:strcv.Jens}
\frac{1}{(2\sigma)^{\frac{\nu}{2}}} \bbE\left[\left(\|\nabla f(x_k)\|_{1-\alpha,2}^*\right)^\nu \right] \geq \left(\bbE\left[ f(x_k)-f^*\right] \right)^{\nu/2} = A_{k}^{\frac{\nu}{2}}
\end{equation}
The main bound \eqref{eq:recur-str-cvx} is secured by twice applying the $(\alpha,q)$-Norm Duality Equivalence Lemma (Lemma \ref{lemma:duality}) to connect the recurrence inequality \eqref{eq:cvx.recur} and the PL-derived bound \eqref{eq:strcv.Jens},
\begin{align*}
A_k- A_{k+1}&\stackrel{\eqref{eq:cvx.recur}}\geq \frac{1}{\nu S_\alpha}\bbE\left[\|\nabla f(x_k)  \|_{\alpha+1-\nu,\nu}^\nu\right] \\
&\stackrel{\text{Lemma }\ref{lemma:duality}}\geq \frac{1}{\nu S_\alpha}\cdot \left(\frac{\displaystyle \min_{1\leq i\leq m}L_i^{\frac{\alpha+1-\nu}{\nu}-\frac{\alpha-1}{2}}}{m^{\frac{1}{2}-\frac{1}{\nu}}} \right)\cdot \bbE\left[\|\nabla f(x_k)\|_{\alpha-1,2}^\nu\right]\\
&\stackrel{\text{Lemma }\ref{lemma:duality}}=\frac{1}{\nu S_\alpha}\cdot \left(\frac{\displaystyle \min_{1\leq i\leq m}L_i^{\frac{(\alpha+1)(2-\nu)}{2\nu}}}{m^{\frac{1}{2}-\frac{1}{\nu}}} \right)\cdot \bbE\left[\left(\|\nabla f(x_k)\|_{1-\alpha,2}^*\right)^\nu\right]\\
&\stackrel{\eqref{eq:strcv.Jens}}\geq \frac{1}{\nu S_\alpha}\cdot \left(\frac{\displaystyle \min_{1\leq i\leq m}L_i^{\frac{(\alpha+1)(2-\nu)}{2\nu}}}{m^{\frac{1}{2}-\frac{1}{\nu}}} \right)\cdot \left[(2\sigma)^{\frac{\nu}{2}}A_{k}^{\frac{\nu}{2}}\right].
\end{align*}
Now, we are prepared to prove the theorem's two constituent parts.\\

\noindent 1. If $\nu=2$, then the main recurrence inequality \eqref{eq:recur-str-cvx} becomes
\[
A_{k+1}\leq A_{k} - A_k^{} \cdot \left( \frac{\sigma}{S_{\alpha}}\right) = A_k\left(1-\frac{\sigma}{S_\alpha} \right),
\]
which by backward induction is equivalent to our desired bound,
\begin{align}\label{eq:str-cvx.nu=2}
\bbE[f(x_k)] -f^* = A_k\leq \left(1-\frac{\sigma}{S_\alpha} \right)^k\cdot  A_0&=\left(1-\frac{\sigma}{S_\alpha} \right)^k \cdot \left[f(x_0) -f^* \right]\\
&\leq \left(1-\frac{\sigma}{S_\alpha} \right)^k \cdot \frac{S_\alpha^{\frac{1}{\nu-1}}(\nu-1)R^{\frac{\nu}{\nu-1}}}{\nu^{\frac{\nu-2}{\nu}}2^{\frac{1}{\nu-1}}},\nonumber
\end{align}
where we have applied Lemma \ref{lemma:nesterov-holder-cvx-str} in the first line.\\

\noindent 2. The $\nu>2$ result requires a verification that is as straightforward as, but more tedious than, that of 1. Applying the Technical Recurrence Lemma (Lemma \ref{lemma:recurrence}) with  $r=\nu/2$, and\\ $\theta=\frac{(2\sigma)^{\frac{\nu}{2}}}{\nu S_\alpha m^{\frac{\nu-2}{2\nu}}}\cdot \displaystyle \min_{1\leq i\leq m} L_i^{\frac{(\alpha+1)(2-\nu)}{2\nu}}$ we see
\begin{equation}\label{eq:rate-for-interpolation}
\mathbb{E}[f(x^{k})]- f^*\leq \frac{f(x^{0})- f^*}{\left[1+\left(\frac{\nu-2}{2}\right)\cdot \frac{(2\sigma)^{\frac{\nu}{2}}}{\nu S_\alpha m^{\frac{\nu-2}{2\nu}}}\cdot \displaystyle \min_{1\leq i\leq m} L_i^{\frac{(\alpha+1)(2-\nu)}{2\nu}}\cdot\left(f(x^{0})- f^*\right)^{\frac{\nu-2}{2}}\cdot k\right]^{\frac{2}{\nu-2}}}.
\end{equation}
This intermediate form of our convergence rate will facilitate the proof of our later convergence rate interpolation result (Corollary \ref{cor:str-cvx.interpolation}) so we have labeled it. 

For now though, we focus on processing this expression of the rate into its final form. The first step is to simply re-arrange this to
\[
\mathbb{E}[f(x^{k})]- f^*\leq \frac{(2\nu S_{\alpha})^{\frac{2}{\nu-2}}m^{\frac{1}{\nu}}(f(x^0)-f^*)}{\left[2\nu S_\alpha m^{\frac{\nu-2}{2\nu}}+ (\nu-2)(2\sigma)^{\frac{\nu}{2}}\displaystyle \min_{1\leq i\leq m} L_i^{\frac{(\alpha+1)(2-\nu)}{2\nu}}\cdot\left(f(x^{0})- f^*\right)^{\frac{\nu-2}{2}}k\right]^{\frac{2}{\nu-2}}}.
\]
By factoring $f(x^{0})- f^*$ out of both the numerator and denominator, and by applying Lemma \ref{lemma:nesterov-holder-cvx-str} to the previous expression, its right-hand side simplifies to
\begin{align*}\label{eq:rate-for-interpolation}
 &= \frac{(2\nu S_{\alpha})^{\frac{2}{\nu-2}}m^{\frac{1}{\nu}}}{\left[\frac{\nu S_\alpha m^{\frac{\nu-2}{2\nu}}}{(f(x^0)-f^*)^{\frac{\nu-2}{2}}}+ (\nu-2)(2\sigma)^{\frac{\nu}{2}}\displaystyle \min_{1\leq i\leq m} L_i^{\frac{(\alpha+1)(2-\nu)}{2\nu}}k\right]^{\frac{2}{\nu-2}}}\\
&\stackrel{\text{Lemma }\ref{lemma:nesterov-holder-cvx-str}}\leq \frac{(2\nu S_{\alpha})^{\frac{2}{\nu-2}}m^{\frac{1}{\nu}}}{\left[\frac{\nu S_\alpha m^{\frac{\nu-2}{2\nu}}}{\left( \frac{S_{\alpha}^{\frac{\nu-2}{2(\nu-1)}} (\nu-1)^{\frac{\nu-2}{2}} R^{\frac{\nu(\nu-2)}{2(\nu-1)}}}{\nu^{\frac{(\nu-2)^2}{2\nu}}2^{\frac{\nu-2}{2(\nu-1)}}}\right)}+ (\nu-2)(2\sigma)^{\frac{\nu}{2}}\displaystyle \min_{1\leq i\leq m} L_i^{\frac{(\alpha+1)(2-\nu)}{2\nu}}k\right]^{\frac{2}{\nu-2}}}\\
&= \frac{(2\nu S_{\alpha})^{\frac{2}{\nu-2}}m^{\frac{1}{\nu}}}{\left[\frac{2^{\frac{\nu-2}{2(\nu-1)}}m^{\frac{\nu-2}{2\nu}} S_{\alpha}^{\frac{\nu-1}{\nu}}\nu^{\frac{\nu^2-2\nu+4}{2\nu}}         
}{(\nu-1)^{\frac{\nu-2}{2}}R^{\frac{\nu(\nu-2)}{2(\nu-1)}}}+ (\nu-2)(2\sigma)^{\frac{\nu}{2}}\displaystyle \min_{1\leq i\leq m} L_i^{\frac{(\alpha+1)(2-\nu)}{2\nu}}k\right]^{\frac{2}{\nu-2}}}\\
&= \frac{(2\nu S_{\alpha})^{\frac{2}{\nu-2}}m^{\frac{1}{\nu}}(\nu-1)R^{\frac{\nu}{\nu-1}}}{\left[2^{\frac{\nu-2}{2(\nu-1)}}m^{\frac{\nu-2}{2\nu}} S_{\alpha}^{\frac{\nu-1}{\nu}}\nu^{\frac{\nu^2-2\nu+4}{2\nu}}+ R^{\frac{\nu(\nu-2)}{2(\nu-1)}}(\nu-1)^{\frac{\nu-2}{2}}(\nu-2)(2\sigma)^{\frac{\nu}{2}}\displaystyle \min_{1\leq i\leq m} L_i^{\frac{(\alpha+1)(2-\nu)}{2\nu}}k\right]^{\frac{2}{\nu-2}}}
\end{align*}

\end{proof}

We now make two crucial comparisons for the rates above. First, it is noteworthy that when $\nu=2$ in the strongly-convex regime, we recover the same linear rate as that in \cite{Nesterov12}. Second, our strongly convex sublinear rate in the $\nu>2$ setting, $\cO(k^{-\frac{2}{\nu-2}})$, is indeed faster than the $\cO(k^{-\frac{1}{\nu-1}})$ rate occurring in the merely convex case. 

To conclude this article, we demonstrate that, in the strongly convex case, when $\nu\to 2$ the intermediate form \eqref{eq:rate-for-interpolation} of the strongly convex sublinear rate of Theorem \ref{thm.conv.rate.strong-convex} converges to its $\nu=2$ linear rate of convergence. 


\begin{corollary}\label{cor:str-cvx.interpolation}
\emph{(Interpolation of Linear and Sublinear Rate)} In the strongly convex setting of Theorem \ref{thm.conv.rate.strong-convex}, if $\nu\to 2$ then the sublinear rate converges to a linear rate of convergence. More formally, for the convergence bound
\[
\mathbb{E}[f(x^{k})]- f^*\leq \frac{f(x^{0})- f^*}{\left[1+\left(\frac{\nu-2}{2}\right)\cdot \frac{(2\sigma)^{\frac{\nu}{2}}}{\nu S_\alpha m^{\frac{\nu-2}{2\nu}}}\cdot \displaystyle \min_{1\leq i\leq m} L_i^{\frac{(\alpha+1)(2-\nu)}{2\nu}}\cdot\left(f(x^{0})- f^*\right)^{\frac{\nu-2}{2}}\cdot k\right]^{\frac{2}{\nu-2}}},
\]
for $k\geq 0$, we observe the limiting result
\begin{multline*}
\lim_{\nu\to 2^+}\frac{f(x^{0})- f^*}{\left[1+\left(\frac{\nu-2}{2}\right)\cdot \frac{(2\sigma)^{\frac{\nu}{2}}}{\nu S_\alpha m^{\frac{\nu-2}{2\nu}}}\cdot \displaystyle \min_{1\leq i\leq m} L_i^{\frac{(\alpha+1)(2-\nu)}{2\nu}}\cdot\left(f(x^{0})- f^*\right)^{\frac{\nu-2}{2}}\cdot k\right]^{\frac{2}{\nu-2}}}\\
\leq\left(1-\frac{\sigma}{2S_\alpha} \right)^k\cdot \left(f(x^{0})- f^*\right)
\end{multline*}
\end{corollary}

\begin{proof}
The form of the sublinear convergence bound here was established by the immediately preceeding theorem in equation \eqref{eq:rate-for-interpolation}. The use of said rate in the proof of this corollary was foreshadowed there. 

To prove our main limit result, suppose for a moment that
\begin{equation}\label{eq:interp.limit}
\lim_{x \to 0^+} \left[1+g(x)\cdot x\right]^{-\frac{1}{x}}= e^{-g(0)}
\end{equation}
holds for any continuously differentiable $g:[0,\infty)\to[0,\infty)$ such that $g(0)>0$. Restating our sublinear convergence rate, we see that
\[
\frac{\mathbb{E}[f(x^{k})]- f^*}{f(x^0)-f^*}\leq 
 \left[ 1+g\left(\frac{\nu-2}{2}\right)\cdot \left(\frac{\nu-2}{2}\right)\right]^{\frac{2}{2-\nu}}
\]
where 
\[
g(x) = \frac{(2\sigma)^{x+1}}{(2x+2)S_\alpha m^{\frac{x}{2x+2}}}\cdot \min_{1\leq i\leq m}L_i^{-\frac{(\alpha+1)x}{2x+2}} \cdot A_0^{x}\cdot k.
\]
Observe that $g$ is continuous differentiable and $g(0)=\sigma/S_\alpha>0$. Thus, it follows that
\begin{align*}
\frac{\mathbb{E}[f(x^{k})]- f^*}{f(x^0)-f^*}\leq & 
\lim_{\nu\to 2^+} \left[ 1+g\left(\frac{\nu-2}{2}\right)\cdot \left(\frac{\nu-2}{2}\right)\right]^{\frac{2}{2-\nu}}\\
\stackrel{\eqref{eq:interp.limit}}= &e^{-g(0)}=\left(e^{-\frac{\sigma}{S_\alpha}}\right)^k \\
\leq & \left(\frac{1}{1+\frac{\sigma}{S_\alpha}}\right)^k \leq \left(1-\frac{\sigma}{2S_\alpha} \right)^k
\end{align*}
where we have used in the last line the standards inequalities $e^{x}\geq 1+x$ and $(1+x)^{-1}<(1-x/2)$, for any nonegative $x$, respectively. Thus, we only need to prove \eqref{eq:interp.limit} to finish the proof.

The proof is a simply straightforward computation:
\begin{align*}
\lim_{x\to 0^+}\left[1+g(x)\cdot x\right]^{-\frac{1}{x}}&=\lim_{x\to 0^+}\exp\left(-\frac{\ln \left[1+g(x)\cdot x\right]}{x}\right)\\
&=\exp\left(-\lim_{x\to 0^+}\frac{\ln \left[1+g(x)\cdot x\right]}{x}\right)\\
&=\exp\left(-\lim_{x\to 0^+}\frac{g(x)+x\cdot g'(x)}{1+x\cdot g(x)}\right)\\
&=\exp(-g(0))
\end{align*}
where continuity of $x\mapsto e^x$ is used in the second line, l'H\^opital's rule is used in the third line, and the definition of $g$ is used in the final line.
\end{proof}

\bibliographystyle{plain}


\clearpage

\appendix

\section{Technical Appendix}
\subsection{Proof of Lemma \ref{lemma:duality}}\label{app:duality}

In this section of the technical appendix, we prove Lemma \ref{lemma:duality}.\\

\noindent 1. We begin by choosing $x,y\in\bbR^d$. The inequality is trivial if $x=0$ so we assume $x\neq 0$. By the standard and $p$-norm versions of the Cauchy-Schwarz inequality, we compute

\begin{multline*}
\ip{x}{y}=\sum_{i=1}^m\ip{P_i x}{P_i y}\leq\sum_{i=1}^m \|P_i x\|\|P_i y\|
=\sum_{i=1}^m \left(L_i^{\alpha/p}\|P_i x\|\right)\left(\frac{\|P_i y\|}{L_i^{\alpha/p}}\right)\\
=\ip{\left(L_1^{\alpha/p}\|P_1 x\|,\ldots,L_1^{\alpha/p}\|P_1 x\|\right)}{\left(\frac{\|P_1 y\|}{L_1^{\alpha/p}},\ldots,\frac{\|P_m y\|}{L_m^{\alpha/p}}\right)}\\
\leq\left\|\left(L_1^{\alpha/p}\|P_1 x\|,\ldots,L_m^{\alpha/p}\|P_m x\|\right)\right\|_p\left\|\left(\frac{\|P_1 y\|}{L_1^{\alpha/p}},\ldots,\frac{\|P_m y\|}{L_m^{\alpha/p}}\right)\right\|_q=\|x\|_{\alpha,p}\|y\|_{-\alpha\frac{p}{q},q}
\end{multline*}
By the standard Cauchy-Schwarz inequality, the first inequality is obtained with equality if and only if one of $P_i x$ and $P_i y$ is a scalar multiple of the other for each $i=1,\ldots,m$. By the Cauchy-Schwarz inequality for $p$-norms,  and our assumption that $x\neq 0$, the second inequality obtains equality if and only if there is some $c\in\bbR$ such that
\[
c^q\cdot \left(L_1^{\alpha}\|P_1 x\|^p,\ldots,L_m^{\alpha}\|P_m x\|^p\right)=\left(\frac{\|P_1 y\|^q}{L_1^{\alpha q/p}},\ldots,\frac{\|P_m y\|^q}{L_m^{\alpha q/p}}\right).
\]
Assuming both inequalities hold then, we conclude  there are $c_1,\ldots,c_m,c\in\bbR$ such that $P_iy=c_i \cdot P_ix$ and $c^q\cdot L_i^{\alpha}\|P_i x\|^p=\frac{\|P_i y\|^q}{L_i^{\alpha q/p}}$ for $i=1,\ldots,m$. Fixing $1\leq i\leq m$, and combining the equalities, we see that $c^q\cdot L_i^{\alpha}\|P_i x\|^p=c_i^q\cdot \frac{\|P_i x\|^q}{L_i^{\alpha q/p}}$, so
\[
c_i=c\cdot L_i^{\alpha\left(1-\frac{q}{p}\right)}\|P_i x\|^{\frac{p}{q}-1}=c\cdot L_i^\alpha \|P_i x\|^{p-2}
\]
where we use the definition of $q$ as $p$'s H{\"o}lder conjugate to produce the second equality. This completes the proof of 1.\\

\noindent 2. Given $x\in\bbR^d$, consider the vector $(\|P_j x\|_j)_{j=1}^{m}$. For any $p\geq 2$, the norm equivalence inequality yields
\[
\|x\|_{\alpha,p} = \left(\sum_{j=1}^m L_j^\alpha\|P_jx\|_2^p\right)^{\frac{1}{p}} = \left(\sum_{j=1}^m \|L_j^{\alpha/p}\cdot P_jx\|_2^p\right)^{\frac{1}{p}} \leq \left(\sum_{j=1}^m \|L_j^{\alpha/p}\cdot P_jx\|_2^2\right)^{\frac{1}{2}} 
\]
But,
\[
\left(\sum_{j=1}^m \|L_j^{\alpha/p}\cdot P_jx\|_2^2\right)^{\frac{1}{2}} = \left(\sum_{j=1}^m L_j^{2\alpha/p} \|P_jx\|_2^2\right)^{\frac{1}{2}} \leq \|x\|_{\beta,2} \cdot \max_{1\leq j\leq m}L_m^{\frac{\alpha}{p}-\frac{\beta}{2}} 
\]
as, for any $i$, with $1\leq i\leq m$, $L_i^{2\alpha/p} \leq L^{\beta}\left(\max_{1\leq j\leq m} L_m^{\frac{\alpha}{p}-\frac{\beta}{2}}\right)^2$ and we complete the first part of the inequality. The second part is quite similar. By the equivalence norm inequality, 
\[
\|x\|_{\beta,2} = \left(\sum_{j=1}^m L_j^{\beta} \|P_jx\|_2^2\right)^{\frac{1}{2}}  = \left(\sum_{j=1}^m  \|L_j^{\beta/2} \cdot P_jx\|_2^2\right)^{\frac{1}{2}} \leq m^{\frac{1}{2}-\frac{1}{p}}\cdot \left(\sum_{j=1}^m  \|L_j^{\beta/2} \cdot P_jx\|_2^p\right)^{\frac{1}{p}}
\]
But
\[
\left(\sum_{j=1}^m  \|L_j^{\beta/2} \cdot P_jx\|_2^p\right)^{\frac{1}{p}} = \left(\sum_{j=1}^m  L_j^{\beta p/2}\|P_jx\|_2^p\right)^{\frac{1}{p}} \leq \|x\|_{\alpha,p}\cdot \max_{1\leq j\leq m}L_{j}^{\frac{\beta}{2}-\frac{\alpha}{p}} = \frac{\|x\|_{\alpha,p}}{\displaystyle \min_{1\leq j\leq m}L_{j}^{\frac{\alpha}{p}-\frac{\beta}{2}}}
\]
and we are done.

\subsection{Proof of Lemma \ref{lemma:nesterov-holder-cvx-str}}\label{app:level-set}

Suppose we are able to prove for all $x,y\in\bbR^d$ that 
\begin{equation}\label{eq:holder-modified}
f(y)\leq f(x)+\ip{\nabla f(x)}{y-x}+\left(\frac{\nu S_\alpha}{2}\right)^{\frac{1}{\nu-1}}\cdot\left(\frac{\nu-1}{\nu}\right)\cdot\|u\|_{(1+\alpha-\nu)(1-\nu),\frac{\nu}{\nu-1}}^{\frac{\nu}{\nu-1}}
\end{equation}
Then $x^*$, the first-order condition $\nabla f(x^*)$ naturally holds so \eqref{eq:holder-modified} implies
\begin{align*}
f(x)-f(x^*)&\leq \ip{\nabla f(x^*)}{x-x^*}+\left(\frac{\nu S_\alpha}{2}\right)^{\frac{1}{\nu-1}}\cdot\left(\frac{\nu-1}{\nu}\right)\cdot\|x-x^*\|_{(1+\alpha-\nu)(1-\nu),\frac{\nu}{\nu-1}}^{\frac{\nu}{\nu-1}}\\
&=\left(\frac{\nu S_\alpha}{2}\right)^{\frac{1}{\nu-1}}\cdot\left(\frac{\nu-1}{\nu}\right)\cdot\|x-x^*\|_{(1+\alpha-\nu)(1-\nu),\frac{\nu}{\nu-1}}^{\frac{\nu}{\nu-1}}\\
&\leq\left(\frac{\nu S_\alpha}{2}\right)^{\frac{1}{\nu-1}}\cdot\left(\frac{\nu-1}{\nu}\right)\cdot R(x)_{(1+\alpha-\nu)(1-\nu),\frac{\nu}{\nu-1}}^{\frac{\nu}{\nu-1}}
\end{align*}
and this completes the proof. Thus, it suffices to prove \eqref{eq:holder-modified}. Suppose that 
\begin{equation}\label{eq:holder-cocoercive}
\frac{2}{\nu S_\alpha}\|\nabla f(x)-\nabla f(y)\|_{1+\alpha-\nu,\nu}^{\nu-1}\leq \|x-y\|_{(1+\alpha-\nu)(1-\nu),\frac{\nu}{\nu-1}}
\end{equation}
holds for all $x,y\in\bbR^d$. Then, given $u\in\bbR^d$, the integral formulation of the mean value theorem states
\[
f(x+u)-f(x)=\int_0^1\ip{\nabla f(x+tu)}{u}dt,
\]
Thus, the Cauchy-Schwartz inequality and our previous inequality imply
\begin{align*}
f(x+u)-f(x)-\ip{\nabla f(x)}{u}&=\int_0^1\ip{\nabla f(x+tu)-\nabla f(x)}{u}dt\\
&\leq\int_0^1 \|\nabla f(x+tu)-\nabla f(x)\|_{1+\alpha-\nu,\nu}\|u\|_{(1+\alpha-\nu)(1-\nu),\frac{\nu}{\nu-1}}dt\\
&\leq\int_0^1 \left(\frac{\nu S_\alpha}{2}\|tu\|_{(1+\alpha-\nu)(1-\nu),\frac{\nu}{\nu-1}}\right)^{\frac{1}{\nu-1}}\|u\|_{(1+\alpha-\nu)(1-\nu),\frac{\nu}{\nu-1}}dt\\
&=\left(\frac{\nu S_\alpha}{2}\right)^{\frac{1}{\nu-1}}\|u\|_{(1+\alpha-\nu)(1-\nu),\frac{\nu}{\nu-1}}^{\frac{\nu}{\nu-1}}\int_0^1 t^{\frac{1}{\nu-1}}dt\\
&=\left(\frac{\nu S_\alpha}{2}\right)^{\frac{1}{\nu-1}}\cdot\left(\frac{\nu-1}{\nu}\right)\cdot\|u\|_{(1+\alpha-\nu)(1-\nu),\frac{\nu}{\nu-1}}^{\frac{\nu}{\nu-1}},
\end{align*}
so taking $u=y-x$ in \eqref{eq:holder-cocoercive} completes the proof. Consequently, we now need only prove the H{\"o}lderian co-coercivity condition \eqref{eq:holder-cocoercive}, which we do presently.

Given $y\in\bbR^d$, the function $x\mapsto\phi(x):=f(x)-f(y)-\ip{\nabla f(y)}{x-y}$ is readily seen to be H{\"o}lder block smooth because $f$. Moreover, $\phi$ has the same block H{\"o}lder smoothness constants and $\nabla \phi(x)=\nabla f(x)-\nabla f(y)$. Thus, by the Block H{\"o}lder Descent Lemma (Lemma \ref{lemma:block_descent}),
\begin{multline*}
f(x)-f(y)-\ip{\nabla f(y)}{y-x}= \phi(y)-\min_{z\in\bbR^d}\phi(z)\geq\max_{1\leq i\leq m}\left[\frac{1}{\nu L_i^{\nu-1}}\|\nabla_i f(x)-\nabla_i f(y)\|^\nu\right]\\
\geq \frac{1}{\nu S_{\alpha}}\sum_{i=1}^m\frac{L_i^\alpha}{L_i^{\nu-1}}\|\nabla_i f(x)-\nabla_i f(y)\|^\nu
= \frac{1}{\nu S_{\alpha}}\|\nabla f(x)-\nabla f(y)\|_{1+\alpha-\nu,\nu}^\nu,
\end{multline*}
which we may restate as
\[
f(x)\geq f(y)+\ip{\nabla f(y)}{y-x}+\frac{1}{\nu S_{\alpha}}\|\nabla f(x)-\nabla f(y)\|_{1+\alpha-\nu,\nu}^\nu.
\]
Adding this inequality to its analogue with the roles of $x$ and $y$ reversed, we see produce
\[
\frac{2}{\nu S_\alpha}\|\nabla f(x)-\nabla f(y)\|_{1+\alpha-\nu,\nu}^\nu\leq\ip{\nabla f(x)-\nabla f(y)}{x-y}.
\]
By Cauchy-Schwarz, we then see that
\[
\frac{2}{\nu S_\alpha}\|\nabla f(x)-\nabla f(y)\|_{1+\alpha-\nu,\nu}^\nu\leq\|\nabla f(x)-\nabla f(y)\|_{1+\alpha-\nu,\nu}\|x-y\|_{(1+\alpha-\nu)(1-\nu),\frac{\nu}{\nu-1}},
\]
or equivalently
\[
\frac{2}{\nu S_\alpha}\|\nabla f(x)-\nabla f(y)\|_{1+\alpha-\nu,\nu}^{\nu-1}\leq \|x-y\|_{(1+\alpha-\nu)(1-\nu),\frac{\nu}{\nu-1}}.
\]
Given $u\in\bbR^d$, the integral formulation of the mean value theorem states
\[
f(x+u)-f(x)=\int_0^1\ip{\nabla f(x+tu)}{u}dt,
\]
so the Cauchy-Schwarz inequality and the previous inequality imply
\begin{align*}
f(x+u)-f(x)-\ip{\nabla f(x)}{u}&=\int_0^1\ip{\nabla f(x+tu)-\nabla f(x)}{u}dt\\
&\leq\int_0^1 \|\nabla f(x+tu)-\nabla f(x)\|_{1+\alpha-\nu,\nu}\|u\|_{(1+\alpha-\nu)(1-\nu),\frac{\nu}{\nu-1}}dt\\
&\leq\int_0^1 \left(\frac{\nu S_\alpha}{2}\|tu\|_{(1+\alpha-\nu)(1-\nu),\frac{\nu}{\nu-1}}\right)^{\frac{1}{\nu-1}}\|u\|_{(1+\alpha-\nu)(1-\nu),\frac{\nu}{\nu-1}}dt\\
&=\left(\frac{\nu S_\alpha}{2}\right)^{\frac{1}{\nu-1}}\|u\|_{(1+\alpha-\nu)(1-\nu),\frac{\nu}{\nu-1}}^{\frac{\nu}{\nu-1}}\int_0^1 t^{\frac{1}{\nu-1}}dt\\
&=\left(\frac{\nu S_\alpha}{2}\right)^{\frac{1}{\nu-1}}\cdot\left(\frac{\nu-1}{\nu}\right)\cdot\|u\|_{(1+\alpha-\nu)(1-\nu),\frac{\nu}{\nu-1}}^{\frac{\nu}{\nu-1}}.
\end{align*}
Taking $u=y-x$ completes the proof.

\end{document}